\documentclass[reqno,12pt]{amsart}
\usepackage{a4wide,color,eucal,enumerate,mathrsfs}
\usepackage[normalem]{ulem}
\usepackage{amsmath,amssymb,epsfig,amsthm}
\usepackage[latin1]{inputenc}
\usepackage{psfrag}
\usepackage{hyperref,cleveref,tikz}
\usepackage{booktabs,longtable,array}
\usetikzlibrary{arrows.meta,positioning,calc}
\usepackage{epigraph}

\newcolumntype{P}[1]{>{\raggedright\arraybackslash}p{#1}}

\numberwithin{equation}{section}

\newtheorem{theorem}{Theorem}[section]

\newtheorem{lemma}[theorem]{Lemma}

\theoremstyle{remark}

\newcommand{\R}{\mathbb{R}}
\newcommand{\Om}{\Omega}

\newcommand{\diver}{\operatorname{div}}

\title[convexity of level sets for elliptic equations]
{A short survey on the convexity of level sets for divergence-form elliptic equations}
\author{Yi Ru-Ya Zhang}
\date{\today}

\address{State Key Laboratory of Mathematical Sciences, Academy of Mathematics and Systems Science, Chinese Academy of Sciences, Beijing 100190, China}
\address{Institute of mathematics, Academy of Mathematics and Systems Science, the Chinese Academy of Sciences, Beijing 100190, China}
\email{yzhang@amss.ac.cn}

\thanks{The author is funded by the National Key R\&D Program of China (Grant No. 2025YFA1018400 \& No. 2021YFA1003100), NSFC Grant No. 12288201 \& No. 12571128, the Chinese Academy of Sciences, and CAS Project for Young Scientists in Basic Research, Grant No. YSBR-031.}


\begin{document}

\begin{abstract}
Let $\Omega\subset\mathbb{R}^n$ be a bounded, smooth and (uniformly) convex domain, and let $u$ be a positive solution to an elliptic boundary value problem with zero Dirichlet data. A classical problem in the analysis of partial differential equations is to determine whether the convexity of $\Omega$ is inherited by the superlevel (or sublevel) sets of $u$. 

This survey considers divergence-form equations in bounded convex domains. For torsion functions, first eigenfunctions, and semilinear equations with sublinear nonlinearities, concavity of suitable transforms of $u$ gives positive results. Two further methods, i.e. envelope constructions and constant-rank theorems for level-set curvature, often provide alternative routes to convexity and related properties. These results are complemented by counterexamples, showing that in general convexity of the domain and ellipticity of the operator do not guarantee quasiconvexity for every solution. We present these counterexamples through a shared underlying idea, and describe the distinct ways in which different equations implement this idea.
\end{abstract}
\maketitle

\section{Introduction}\label{sec:framework}

\subsection{The geometric question}

Convexity properties of solutions form a classical meeting point of elliptic partial differential equations and convex geometry.  Even when a boundary value problem is well posed on a convex domain, it is rarely reasonable to expect the solution itself to be convex or concave.  A more natural question concerns its level sets.  

Suppose that $\Omega\subset\mathbb R^n$ is bounded and convex, and $u>0$ in $\Omega$ with $u=0$ on $\partial\Omega$, we write
$$
 \Omega_t=\{x\in\Omega:u(x)>t\},
 \qquad 0<t<\max_{\overline\Omega}u.
$$
The function $u$ is called \emph{quasiconcave} when every set $\Omega_t$ is convex. One defines the quasiconvexity in a similar manner for sublevel sets.   This property is invariant under increasing changes of the dependent variable and is therefore weaker, and often more robust, than ordinary concavity of functions.

Several benchmark theorems suggest how much structure can lie behind quasiconcavity.  For the torsion problem, square-root concavity was proved in the plane by Makar--Limanov \cite{ML1971} and in arbitrary dimension by Kennington \cite{Kennington1985}.  Brascamp and Lieb \cite{BL1976} established the log-concavity of the first Dirichlet eigenfunction by showing that log-concavity is preserved under the heat flow. In particular,  Kennington's power-concavity theorem \cite[Theorems 3.3 and 4.2]{Kennington1985} gives concavity of $u^{(1-\gamma)/2}$ for $-\Delta u=ku^\gamma$, $0\leq\gamma<1$.  These results prove more than convexity of level sets: Indeed, they identify a transformed function whose Hessian has a sign.  Subsequently, the  constant-rank results of Caffarelli--Friedman \cite{CF1985} and Korevaar--Lewis \cite{KL1987} then explain how an already established Hessian sign can become even strict. Actually, it was conjectured earlier by Lions \cite[Remark 3]{L1981} that, "for general $f$, the (super)level sets of any solution $u\in C^2(\overline{\Omega})$ of  
\begin{equation}\label{eq:semilinear-equ}
 -\Delta u=f(u) \quad  \text{ in }\ \Omega, \qquad  u=0\quad \text{ on } \ \partial\Omega, \qquad  u>0,
\end{equation}
are convex."

However, there are important equations for which no useful global concavity transform is known.  Usually two other traditions address this situation.  The first one studies several points on the same level, or equivalently, the quasiconcave envelope of solutions. It converts the geometry of supporting hyperplanes into a viscosity comparison.  The second one studies the second fundamental form of a single regular level and applies a microscopic maximum principle to its smallest-rank modes.  Caffarelli--Spruck \cite {CS1982}, Korevaar \cite{K1990}, and Colesanti--Salani  \cite{CS2003} developed representative versions of these methods.  The three positive viewpoints, i.e.  a transformed Hessian, a finite equal-level contact, and an infinitesimal curvature matrix, are closely related, but they are not generally interchangeable.  One of the purposes of this survey is to make their different assumptions and conclusions visible in the formulas.

The positive part of the theory is not the whole story.  A celebrated work by Wang \cite{W2014} established a smooth uniformly convex domain for which the solution of the constant-mean-curvature Dirichlet problem has a nonconvex level set.  This result directly implies that energy minimality does not guarantee the quasiconvexity or quasiconcavity of the corresponding minimizers. On a separate note,  Hamel, Nadirashvili, and Sire \cite{HNS2016} gave an important example by  presenting a semilinear problem with suitably chosen $f$ in \eqref{eq:semilinear-equ} on a convex domain, with two positive solutions of opposite  geometry on their superlevel sets in deliberately chosen branches, answering the problem of Lions negatively.  More recently, the author in \cite{Z2026} showed that even for the minimal and strictly stable  branch, a planar solution may have a nonconvex superlevel set, which answers \cite[Open Problem 3]{B2003} negatively. Subsequently a non-quasiconvex solution for the planar translating mean-curvature equation was  constructed in \cite{LLMQ2026},  via a chosen corrector together with a barrier argument.  These examples change the form of the question.  Convexity of the domain and ellipticity of the operator, even minimality or stability,  are not sufficient data for an individual solution. 

\subsection{Scope and organizing viewpoints}

The main focus is the divergence-form family
\begin{equation}\label{eq:prototype}
 \begin{cases}
 -\operatorname{div}\!\bigl(a(|\nabla u|)\nabla u\bigr)=\mathcal F(u,|\nabla u|)&\text{in }\Omega,\\
 u>0&\text{in }\Omega,\\
 u=0&\text{on }\partial\Omega,
 \end{cases}
\end{equation}
with ellipticity expressed by
$$
 a(q)>0,  \quad \text{ and } \quad  a(q)+qa'(q)>0,
$$
where the first condition controls diffusion perpendicular to the gradient, and the second one controls diffusion parallel to the gradient. 

This class contains the semilinear Laplacian, the $p$-Laplacian (on its regular set), and the prescribed mean-curvature operator.  Existence and regularity are assumed in the form required by the individual results; we refer to Serrin  \cite{S1969} on the classical quasilinear Dirichlet background.  When an operator is degenerate at $|\nabla u|=0$, differential statements are usually understood on the regular set
$$
 \mathcal R(u)=\{x\in\Omega:|\nabla u(x)|>0\},
$$
together with weak or viscosity formulations applied  appropriately in $\Omega$. 
Moreover, in this manuscript, we often adopt superlevels as the default orientation. When a cited reference is written for a negative solution or for sublevels, the sign change and the corresponding reversal of orientation are usually stated explicitly in this manuscript.

As indicated by the geometry, local and global arguments are treated separately below. In fact, on a regular level,  we can orient $\partial\Omega_t$ by the outer normal $ N_t=-\frac{\nabla u}{|\nabla u|}$ of the superlevel set, and let $\mathcal B_t$ denote its second fundamental form.  For tangent vectors $\tau,\sigma$, one has
$$
 \mathcal B_t(\tau,\sigma)
 =-\frac{D^2u(\tau,\sigma)}{|\nabla u|}.
$$
In particular, when $u\in C^2(\overline\Omega)$ and the Hopf normal derivative is nonzero, differentiation of the boundary condition gives the elementary boundary identity $ \mathcal  B_0=\mathrm{II}_{\partial\Omega}.$
A uniformly convex boundary ensures the existence of a strictly convex collar neighborhood for sufficiently low superlevel sets. However, this condition does not exclude the degeneration of curvature at an interior regular level, nor does it  regulate the geometry of the critical region at the maximum value; see \cite{HNS2016}. In particular, the presence of a unique nondegenerate critical point is even insufficient to determine the convexity of the level sets in between. For example, in the case of a stable solution $u$  to a semilinear equation (i.e., $a\equiv 1$ in \eqref{eq:prototype}) with a superlinear positive nonlinearity $f$, the uniqueness of the nondegenerate critical point of $u$ in a bounded uniformly convex  domain in the plane was proven in the seminal work \cite{CC1998}, and hence sufficiently high superlevel set of $u$ are uniformly convex. In contrast, \cite{Z2026} constructed a counterexample to the convexity of superlevel sets.

This survey is organized around two complementary principles. Positive results are classified based on the variants of convex structures, upon which the maximum principle operates. Conversely, counterexamples are categorized according to how a strict geometric defect in convexity is made in an exact solution. In the current manuscript, a \emph{branch} means a specified selection of solutions of the same boundary value problem, such as a minimal branch, a constrained minimizer, etc.  The counterexamples are typically realized by moving from an approximate solution  to an exact solution, and throughout this transition one must carry over the strict defect.  Note that, a small PDE residual does not by itself produce a nearby solution, and $C^0$-convergence can preserve a finite midpoint defect but not a sign of curvature.

Convex rings are not the main setting in the current survey.  Their two constant boundary components often provide a nonvanishing-gradient foliation and make a global comparison possible, while a one-boundary convex domain instead terminates at a critical maximum set which might not be connected a priori.  This emphasis is consistent with the classical division of level-set methods, for instance surveyed by Kawohl in \cite[Introduction \& Section III.11]{Kawohl1985}.

The aim of this work is not to present an exhaustive bibliography nor to advance a necessary-and-sufficient classification. Instead, it provides a targeted analysis of a restricted class of mechanisms for which the corresponding calculations allow direct comparison, and it describes the specific conditions under which these mechanisms no longer determine the global geometry of the level sets. This survey aims to briefly summarize the author's viewpoint and current understanding of the topic. Indeed, in the opinion of the author, during the age of AI  it is increasingly important to clearly lay out the reasoning process and understanding, and communicate them to the community, rather than leaving readers to guess them.

\medskip

\noindent{\bf AI disclaimer}: The author used ChatGPT Pro 5.6 along with Overleaf AI to polish the initial manuscript, while the main ideas and interpretations were developed and expressed solely by the author. Specifically, AI tools were used to support the preparation of the figure.

\section{Positive results}

\subsection{Power concavity/convexity and maximum principles}\label{sec:power}

We begin by introducing transformations, namely the composition method.
Let $v=\Phi(u)$ with $\Phi'>0$.  Then
$$
 D^2v=\Phi'(u)D^2u+\Phi''(u)\,Du\otimes Du.
$$
If $\tau\perp Du$, the last rank-one term vanishes as
$$
 D^2v(\tau,\tau)=\Phi'(u)D^2u(\tau,\tau).
$$
Thus, concavity of $v$ implies convexity of every regular superlevel of $u$ for $C^2$-solutions $u$, which is independent of the particular choice of $\Phi$. Moreover,  since concavity is a global property on sets, it also controls the limiting top level.

Let us consider a two-point maximum principle first. For $\lambda\in(0,1)$, let us define the convexity deficit
\begin{equation}\label{eq:convexity-deficit}
 \mathcal C_v(x,y;\lambda)
 =\lambda v(x)+(1-\lambda)v(y)
  -v\bigl(\lambda x+(1-\lambda)y\bigr).
\end{equation}
The function $v$ is concave exactly when $\mathcal C_v\leq0$ for every $x,y$ and $\lambda$.  The phrase \emph{two-point} or \emph{multipoint maximum principle} usually refers to applying the maximum principle to this auxiliary function (or its variant) on a product domain, where $x$ and $y$ are the free points and the third point is their   convex combination.

From the viewpoint of the author, the local calculation behind the convexity maximum principle is, somewhat, independent of the particular choice of the transform.  Indeed, suppose that $\mathcal C_v$ has a positive interior maximum at $(x,y)$ and set
$$
 z=\lambda x+(1-\lambda)y.
$$
Then by the maximality of $\mathcal C_v$, its first-order derivative vanishes at $(x,y)$, implying
\begin{equation}\label{eq:three-gradients}
 Dv(x)=Dv(y)=Dv(z)=:\mathbf {p}\in \mathbb R^n.
\end{equation}
Moreover, for variations $x+s\xi$ and $y+s\eta$, the second derivative test gives
\begin{align}
 &\lambda D^2v(x)[\xi,\xi]
 +(1-\lambda)D^2v(y)[\eta,\eta] \notag\\
 &\qquad
 -D^2v(z)[\lambda\xi+(1-\lambda)\eta,
                    \lambda\xi+(1-\lambda)\eta]
 \leq0;\label{eq:block-second-variation}
\end{align}
see e.g. \cite{K1983}. 
Thus the choice $\xi=\eta=\mathbf {e}_i$ already yields an arithmetic-mean trace inequality.

Kennington in \cite[Theorems 3.1]{Kennington1985} replaces these  variations by variations adapted to the elliptic coefficient matrix, and the resulting scalar comparison involves a harmonic mean.  His Theorem 3.3 then applies that principle to $v=u^\alpha$, including the boundary-exclusion condition needed for zero Dirichlet data.  To a certain extent, this provides part of the conceptual basis underlying the term "two-point maximum principle."

\medskip
\paragraph{\it Torsion functions} We consider some examples, starting with torsion functions. Let
$$
 -\Delta u=1\quad\text{in }\Omega,
 \qquad u=0\quad\text{on }\partial\Omega.
$$
Set $w=\sqrt u$.  Since $u=w^2$,
\begin{equation}\label{eq:sqrt-torsion}
 \Delta w+\frac{|Dw|^2}{w}+\frac1{2w}=0.
\end{equation}
Makar--Limanov  \cite{ML1971} proved square-root concavity in the planar problem, and Kennington \cite[Theorem 4.1]{Kennington1985} gives the $n$-dimensional convex-domain result for constant right-hand side.  Furthermore, the Hopf boundary estimate yields the requisite boundary exclusion. The contact relations \eqref{eq:three-gradients}--\eqref{eq:block-second-variation} allow one to localize an a priori assumed loss, while the $w^{-1}$ dependence in \eqref{eq:sqrt-torsion} provides the harmonic-mean inequality required for the argument.

Concavity is not yet strict concavity.  The missing strictness is where the   constant-rank theorem belongs in the development. 

\begin{theorem}[{\cite{KL1987}}] \label{thm:KL}
Let $v\in C^4(\Omega)$ solve
$$
 \Delta v=F(v,Dv)>0
$$
with $D^2v\geq0$.  Assume that $F\in C^{2,\alpha}, \ 0<\alpha<1$ satisfies
\begin{equation}\label{eq:KL-structure}
 2F_v^2-FF_{vv}\geq0,
\end{equation}
which is equivalent to convexity of $1/F$ in the $v$-variable for fixed gradient.  Then $D^2v$ has constant rank in $\Omega$.
\end{theorem}

This result extended the original planar semilinear mechanism by Caffarelli--Friedman \cite{CF1985}.  To apply it to torsion, set $v=-\sqrt u<0$.  Equation \eqref{eq:sqrt-torsion} becomes
$$
 \Delta v=-\frac{|Dv|^2+1/2}{v}=:F(v,Dv)>0,
 \qquad
 \frac1F=-\frac{v}{|Dv|^2+1/2}.
$$
For fixed $Dv$, the mapping $v \mapsto 1/F$ is affine, and therefore \eqref{eq:KL-structure} is satisfied with equality; in fact, this is usually the standard to find out the transformation that maps $u$ to $v$.  Korevaar--Lewis then deformed a ball through smooth strictly convex domains by a continuity method.  The explicit ball solution starts with $D^2v>0$\footnote{See also \cite{GNN1979} for the convexity of the level sets in balls for general nonlinearity. Note that here one requires not only the convexity of the superlevel sets of $v$, but also the function $v$ itself in the unit ball.}, and boundary estimates give a uniform collar neighborhood along the deformation for which the Hessian is always strictly positive. Then  Theorem~\ref{thm:KL} prevents a first interior rank loss.  Their application  proves strict concavity of $\sqrt u$ in the smooth strictly convex case.

The two proofs have different strengths.  Kennington's maximum principle directly gives global concavity and survives approximation to general convex domains.  Constant rank upgrades an already entered convexity cone to strictness, but it needs a strict anchor and a deformation, or continuity, argument. For further quantitative estimates, we refer to  \cite{MSY2012} and some other subsequent work. 

\medskip
\paragraph{\it The first eigenfunction}

Let $\phi_1>0$ solve
$$
 -\Delta\phi_1=\lambda_1\phi_1\quad\text{in }\Omega,
 \qquad \phi_1=0\quad\text{on }\partial\Omega.
$$
For $v=\log\phi_1$,
$$
 \Delta v+|Dv|^2+\lambda_1=0.
$$
At a contact triple, the gradient term again agrees at all three points by \eqref{eq:three-gradients}.  

Brascamp and Lieb [3, Theorem 6.1] originally proved that $\phi_1$ is log-concave using a heat-semigroup approach derived from applying the Pr\'ekopa--Leindler inequality to a product representation of the solution. Nevertheless, the same conclusion can also be obtained by adapting the argument of Korevaar--Lewis \cite{KL1987}. 

\medskip
\paragraph{\it Sublinear source  terms}
Now suppose
\begin{equation}\label{eq:sublinear-power}
 -\Delta u=ku^\gamma\quad\text{in }\Omega,
 \qquad u>0\quad \text{ in }\ \Omega,
 \qquad u=0\quad \text{ on } \ \partial\Omega,
\end{equation}
where $k>0$ and $0\leq\gamma<1$.  Write $u=v^{\frac 1 \alpha}$.  A direct differentiation gives
$$
 \Delta v
 +\left(\frac1\alpha-1\right)\frac{|Dv|^2}{v}
 +\alpha k\,v^{\frac{\gamma+\alpha-1} \alpha}=0.
$$
The exponent is selected by requiring the last term to have the same $v^{-1}$ dependence on the gradient term:
$$
 \frac{\gamma+\alpha-1}{\alpha}=-1
 \quad\Longleftrightarrow\quad
 \alpha=\frac{1-\gamma}{2}.
$$
Hence
$$
 \Delta v  +\frac{1-\alpha}{\alpha}\frac{|Dv|^2}{v}
 +\frac{\alpha k}{v}=0 \qquad \text{ for }\ 
 v=u^{\frac{1-\gamma} 2}.
$$
Kennington \cite[Theorem 4.2]{Kennington1985} states that, on a bounded convex domain satisfying the interior sphere condition, i.e. every boundary point admits a tangent ball contained in the domain, the positive zero-Dirichlet solution of \eqref{eq:sublinear-power} is $\alpha$-concave with $\alpha=(1-\gamma)/2$.  Therefore, his theorem  gives exactly the conclusion: 
$$
 u^{\frac{1-\gamma} 2}\quad \text{is concave}.
$$
Observe that the endpoint $\gamma=0$ recovers the torsion case. Also when $\gamma\to 1^-$, one has $\alpha\to 0^+$, and the normalized transforms $(u^\alpha-1)/\alpha$ converges to $\log u$ for each fixed $s>0$, formally connecting the power and logarithmic transformations.

\medskip
We finally remark that, the transform concavity proves a structural theorem for every solution covered by its assumptions.  A different positive mechanism, starting from a particularly selected branch,  was given by   Gladiali--Grossi \cite{GG2004}.   The mechanism is based on perturbation together with asymptotic structure, but not an ordinary Hessian or level-curvature continuity method, and its conclusion is specially on branches and is not for every positive solution of the semilinear operator. For further results on constant-rank theorem, see e.g. \cite{CGM2007,BGMX2011} and some other subsequent work. 

\subsection{Quasiconcave envelopes}\label{sec:envelope}

A suitable concavity transform  may not always exist. In particular, it is often necessary to control the concavity/convexity properties of the transformed function  in a collar neighbhorhood of the boundary, while simultaneously ensuring compatibility with additional structural assumptions such as \eqref{eq:KL-structure}.  Alternatively, quasiconcavity may be characterized directly in terms of the function's superlevel sets. This idea was earlier employed, for instance,  in \cite{L1977} (up to a sign convention). 

We define the quasiconcave envelope of $u$ as
$$
 u^*(x)=\sup\{t:x\in\operatorname{conv}\{u\geq t\} \},
$$
where $\operatorname{conv}\{u\geq t\}$ denotes the convex hull of the set $\{u\geq t\}$. 
Then $u^*\geq u$, and $u^*=u$ exactly when every superlevel of $u$ is convex. 
In contrast to the deficit \eqref{eq:convexity-deficit}, the envelope can admit multiple contact points at the same level set, for which the lengths of the associated gradients are not required to coincide. Consequently, the corresponding second-order normal and tangential estimates typically involve distinct weighting factors.

The foundational viscosity idea of using envelope was developed by Alvarez, Lasry and Lions \cite{ALL1997}.
For a more transparent argument, we refer to  \cite{CS2003}, where Colesanti and Salani identified the contact representation in a convex ring with $|Du|>0$ in Proposition 3.1, and their Theorem 3.2 gave  the viscosity Laplacian inequality for $u^*$. Namely, whenever a $C^2$-test function touches $u^*$ from above, its normal and tangential second derivatives satisfy certain lower bounds at the contact point.  The local formulas are useful beyond rings, but their global comparison theorem uses the two boundary components and the nonvanishing-gradient assumption.

To be more specific, let us see a planar calculation. At a regular point of a convexified level in the plane, two contacts already yield the geometry.  Let
$$
 x=\lambda x_1+(1-\lambda)x_2,
 \qquad  u(x_1)=u(x_2)=t,
 \qquad 0<\lambda<1,
$$
where $x_1$ and $x_2$ lie on a common support line of $\{u\geq t\}$.  If $\mathbf e\in \mathbb S^{1}$ is the unit normal pointing toward increasing values of $u$, then
$$
 Du(x_i)=a_i \mathbf {e},  \qquad a_i=|Du(x_i)|>0,
 \qquad i=1,2.
$$

\begin{figure}[ht]
\centering
\begin{tikzpicture}[scale=1.05,>=Latex]
\draw[dashed] (-3.6,0)--(3.6,0) node[right] {common support line};
\draw[thick] plot[smooth] coordinates {(-3.2,1.5) (-2.7,0.45) (-2.2,0) (-1.7,0.45) (-1.2,1.35)};
\draw[thick] plot[smooth] coordinates {(1.0,1.35) (1.5,0.45) (2.0,0) (2.5,0.45) (3.0,1.5)};
\fill (-2.2,0) circle (2pt) node[below] {$x_1$};
\fill (2.0,0) circle (2pt) node[below] {$x_2$};
\fill (-0.1,0) circle (2pt) node[below] {$x$};
\draw[->,thick] (-2.2,0)--(-2.2,1.0) node[left] {$a_1\mathbf e$};
\draw[->,thick] (2.0,0)--(2.0,1.0) node[right] {$a_2\mathbf e$};
\draw[<->] (-2.2,-0.55)--(-0.1,-0.55) node[midway,below] {$1-\lambda$};
\draw[<->] (-0.1,-0.55)--(2.0,-0.55) node[midway,below] {$\lambda$};
\end{tikzpicture}
\caption{Two equal-level contacts with a common support line.  The gradients are parallel, but their magnitudes need not coincide.}
\label{fig:two-contact}
\end{figure}
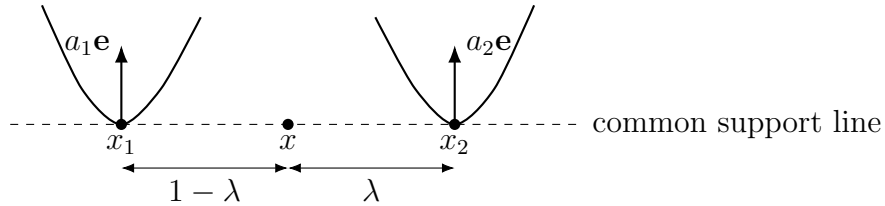

Now translate the support line by a distance $h$ along the direction $\mathbf e$.  If the two contacts move by normal distances $h_1,h_2$, then up to higher order terms,  
$$
 \lambda h_1+(1-\lambda)h_2=h,
 \qquad
 a_1h_1=a_2h_2=a_*h,
$$
and hence
$$
 a_*=\left(\frac\lambda{a_1}+\frac{1-\lambda}{a_2}\right)^{-1}.
$$
This is precisely an interpretation of \cite[Proposition 3.5]{CS2003} in the plane.  Moreover, a second-order Taylor expansion in the common normal gives
$$
 u^*_{ee}(x)
 \geq a_*^3\left(
 \frac\lambda{a_1^3}u_{ee}(x_1)
 +\frac{1-\lambda}{a_2^3}u_{ee}(x_2)
 \right),
$$
which is the planar form of  \cite[Proposition 3.6]{CS2003}.

In addition, tangential estimate comes from convex geometry.  The weighted Minkowski sum of two convex sets is 
$$\lambda K_1+(1-\lambda)K_2=\{\lambda x_1+(1-\lambda)x_2:x_i\in K_i\}.$$  
Suppose that the contacted level curves have positive curvatures $\kappa_1,\kappa_2$. Then their radii of curvature add linearly under this combination, and hence
$$
 \kappa_*
 \leq \left(\frac\lambda{\kappa_1}+\frac{1-\lambda}{\kappa_2}\right)^{-1}.
$$
Since $u_{\tau\tau}(x_i)=-a_i\kappa_i,$
this becomes a tangential Hessian estimate.  Via this idea,  \cite[Proposition 3.7]{CS2003} gives the corresponding viscosity tangential-Laplacian bound, and then \cite[Theorem 3.2]{CS2003} combines it with the normal estimate.

Next one passes from local contact formulas to a PDE comparison. 
For the Laplace problem in a convex ring, Colesanti--Salani prove that the envelope is a viscosity subsolution and then compare it with the unique solution; see  \cite[Theorems 4.2 and 4.4]{CS2003}. Thus the logical chain is 
$$
 \begin{gathered}
 \text{equal-level contacts}
\  \Rightarrow \ 
 \text{normal and tangential estimates}\\
\ \Rightarrow \ 
 u^*\text{ is a subsolution}
\  \Rightarrow \  u^*=u.
 \end{gathered}
$$
The first two arrows are local, while the final arrow needs boundary equality, comparison, uniqueness of the solution, together with control of critical levels.  Especially, in a convex domain with one boundary component, the regular foliation terminates at a maximum, and the ring theorem therefore cannot simply be applied as a globalization result in convex domains.

Furthermore, envelope constructions can be also used to prove concavity of an appropriately transformed solution, as well as convexity of the superlevel sets of the original solution, linking to the power concavity/convexity discussed in the previous subsection. As mentioned above, an early viscosity-based framework was introduced by Alvarez, Lasry and Lions \cite{ALL1997}, and for subsequent developments of this approach we refer, for instance, to \cite{CS2006,BLS2009,BS2013}. These constructions also have applications to Brunn--Minkowski inequalities for variational functionals.  

We end this subsection by a comparison between the constant-rank theorem and the quasiconcave envelopes. 
Let $v$ be a smooth function and $\tau$ a tangent direction. Then the second directional derivative can be recovered as the limit of the associated symmetric second difference:
\begin{equation}\label{eq:coalescence}
 \lim_{s\to0}
 \frac{v(x+s\tau)+v(x-s\tau)-2v(x)}{s^2}
 =D^2v(x)[\tau,\tau].
\end{equation}
In this sense, a direction $\tau$ satisfying $D^2v(x)[\tau,\tau]=0$ may be interpreted as the infinitesimal version of a vanishing symmetric two-point  deficit. By contrast, a finite chord connecting two points on the same level set, or on distinct connected components of a regular level, can reflect global geometric features and may link spatially separated regions.  Such a configuration does not, in general, arise as the limit of  point pairs. This separation between infinitesimal degeneracy and finite two-point contact also partially clarifies why constant-rank theorem  cannot uniformly substitutes for arguments that depend essentially on finite-contact configurations.

\subsection{Microscopic propagation of level curvature}\label{sec:microscopic}

Now we consider the curvature of the level sets.  Suppose that $|Du|>0$ and choose an orthonormal frame with
$$
 \mathbf e_n=\frac{Du}{|Du|},
 \qquad \mathbf e_1,\dots,\mathbf e_{n-1}\perp Du.
$$
Note that
$$
  D_{\mathbf p}\big(a(|\mathbf {p} |)\mathbf {p}\big)\bigg|_{\mathbf {p}=Du}
 =a(|Du|)I+ |Du| a'(|Du|)  \frac{Du}{|Du|}\otimes \frac{Du}{|Du|},
$$
where the tangential eigenvalue is $a(|Du|)$ and the normal eigenvalue is $a(|Du|)+|Du|a'(|Du|)$.  Hence
$$
 \operatorname{div}\big(a(|Du|)Du\big)
 =a(|Du|)\sum_{\alpha=1}^{n-1}u_{\alpha\alpha}
 +\big(a(|Du|)+|Du|a'(|Du|)\big)u_{nn}.
$$
With
$$
 \mathcal B_{\alpha\beta}=-\frac{u_{\alpha\beta}}{|Du|},
 \qquad H=\operatorname{tr}\mathcal B,
$$
the   equation becomes
\begin{equation}\label{eq:curvature-balance}
 a(|Du|)|Du|H-\big(a(|Du|)+|Du|a'(|Du|)\big)u_{nn}=\mathcal F(u,|Du|).
\end{equation}
Thus the operator does not act on the full Hessian isotropically; indeed the tangential curvature and normal second derivative carry different coefficients.

The ratio
$$
 \Theta(q)=\frac{a(q)+qa'(q)}{a(q)}
$$
measures the diffusion from the normal direction to the tangential one.  For the Laplacian, $p$-Laplacian, and prescribed mean curvature, \eqref{eq:curvature-balance} gives respectively
\begin{align}
 qH-u_{nn}&=f(u),\nonumber\\
 q^{p-1}H-(p-1)q^{p-2}u_{nn}&=f(u),\nonumber\\
 \frac{qH}{\sqrt{1+q^2}}-
 \frac{u_{nn}}{(1+q^2)^{3/2}}&=c,\label{eq:mc-balance}
\end{align}
where $q=|Du|$. 

Now we consider the local hodograph. At a regular point, we rotate coordinates so that $u_n=|Du|>0$ and $u_\alpha=0$ for $\alpha<n$.  Use the level value as a coordinate and write
$$
 x_n=v(y,t), \qquad y=(x_1,\dots,x_{n-1}),
 \qquad t=u(x).
$$
We differentiate $u(y,v(y,t))=t$ and obtain 
$$
 v_t=q^{-1},
 \qquad  v_{\alpha\beta}=-\frac{u_{\alpha\beta}}q=\mathcal B_{\alpha\beta},
 \qquad  v_{\alpha t}=-\frac{u_{\alpha n}}{q^2},
 \qquad  v_{tt}=-\frac{u_{nn}}{q^3}.
$$
Thus $\mathcal B=D_y^2v$.  Substitution into \eqref{eq:curvature-balance} yields
$$
 a(q)q\,\operatorname{tr}D_y^2v
 +\big(a(q)+qa'(q)\big)q^3v_{tt}
 =\mathcal F(t,q),
 \qquad q=v_t^{-1}.
$$
This is a local identity at a chosen point on $\{|Du|>0\}$, which does not cross a critical point.

Now we return to  Korevaar's   constant-rank theorem \cite{K1990} on level sets.  It concerns equations that, in an adapted frame, have one coefficient $A(u,q)$ on the tangential trace and another coefficient $B(u,q)$ on the normal second derivative:
\begin{equation}\label{eq:Korevaar-form}
 A(u,q)\left(\Delta u-\frac{u_i u_j}{q^2}u_{ij}\right)
 +B(u,q)\frac{u_i u_j}{q^2}u_{ij}=F(u,q).
\end{equation}
If $u\in C^4$, $Du\neq0$, the regular levels are already locally convex, and
$$
 \left(\sqrt{\frac AB}\right)_{qq}\geq0,
 \qquad  \left(-\frac{F}{Bq^2}\right)_{qq}\leq0,
$$
then the level-set second fundamental forms have constant rank on the connected domain.  This is   stated in \cite[Theorems 1 and 6]{K1990} for locally convex sublevel sets and should be applied for $-u$ here.  In addition, his Remark 2 checks the prescribed $p$-Laplacian and prescribed mean-curvature examples, and Theorem 11 combines the rank theorem with a continuation argument in convex rings.
For the isotropic divergence-form equation \eqref{eq:prototype}, the formal identification in \eqref{eq:Korevaar-form} is
$$
 A=a(q),
 \qquad B=a(q)+qa'(q),
 \qquad F=-\mathcal F(u,q).
$$

Now assume $\mathcal B=D_y^2v\geq0$ on a connected regular region, and let $r$ be its minimum rank.  The standard test function is
$$
 \varphi=\sigma_{r+1}(\mathcal B),
$$
where $\sigma_{r+1}$ is the $(r+1)$-th elementary symmetric function of the  principal curvatures. 
Near a minimum-rank point, one splits the eigenvalues into positive ``good'' and small ``bad'' groups.  Differentiating the PDE and using the structural inequalities, one obtains a local estimate of the form
$$
 \mathcal L\varphi
 \leq C\bigl(\varphi+|D\varphi|\bigr),
$$
where $\mathcal L$ is uniformly elliptic on the regular region.  Since $\varphi\geq0$ and vanishes at a minimum-rank point, the strong minimum principle makes its zero set open, while continuity makes it closed. We also record  \cite{X2011} by Xu, which gives a quasilinear microscopic theorem in this direction under its stated structural condition. See also \cite{HMO2010}. 

The conclusion is then by propagation. If the rank is deficient at a given point, the theorem yields that this deficiency persists under propagation, and it does not generate an additional positive eigenvalue. Consequently, auxiliary inputs, such as a boundary collar argument, are required to rule out the lower-rank alternative a priori. Moreover, the theorem neither provides a mechanism for linking distinct regular components nor permits extension across the critical sets.

\bigskip
Finally, let us briefly summarize these three mechanisms. 
\begin{enumerate}[(i)]
\item Theorem~\ref{thm:KL} concerns $D^2\Phi(u)$ and the concavity/convexity of a transformed function.
\item The envelope formulas  concern separated equal-level contacts and can detect a finite chord which might be invisible to a one-point rank test.
\item Korevaar's theorem concerns the $(n-1)\times(n-1)$ second fundamental form of regular levels and assumes $Du\neq0$.
\end{enumerate}
Equation \eqref{eq:coalescence} somehow relates the last two mechanisms infinitesimally, but it does not see whether  every first loss of quasiconcavity is a discrete set or not in the continuity argument; this is indeed an interesting question when there exists a counterexample to the equation.

\section{Counterexamples for the convexity of level sets}\label{sec:counterexamples}

A counterexample concerning the convexity of level sets is typically constructed by introducing an almost-solution and then establishing two logically independent estimates.  Firstly, one computes a residual
\begin{equation}\label{eq:abstract-residual}
 \mathscr R_\epsilon
 :=\mathscr F_\epsilon[V_\epsilon]
\end{equation}
for an almost-solution $V_\epsilon$.  Secondly, one proves an estimate comparing $V_\epsilon$ with a real solution $u_\epsilon$.  Note that, the smallness of \eqref{eq:abstract-residual} alone gives no such comparison, and it must be converted into an order interval via ordered sub- and supersolutions (or barriers), or inverted through a linearized operator with a positive spectral gap, i.e. the coercivity needed for an error estimate.  Additionally, the comparison norm must be strong enough for the chosen geometric defect.

To illustrate clearly, let us see an elementary example, which gives a transfer from the defect of the quasiconcavity from $V$ to that of $u$. 
\begin{lemma}\label{lem:finite-defect-transfer}
Let $\Omega$ be a convex domain with $P_\pm\in \Omega$ and $M=\frac{P_-+P_+} 2$.  Suppose that, for some $t\in\R$ and $\eta>0$,
\begin{equation}\label{eq:superlevel-transfer}
 V(P_\pm)\geq t+2\eta,
 \qquad
 V(M)\leq t-2\eta,
 \qquad
 \|u-V\|_{L^\infty}\leq\eta.
\end{equation}
Then the superlevel set $\{u>t\}$ is not convex.  Moreover, with all inequalities reversed, the analogous statement holds for a sublevel set.
\end{lemma}
\begin{proof}
Observe that, equation \eqref{eq:superlevel-transfer} gives $u(P_\pm)>t$ and $u(M)<t$. Since $M$ is the midpoint of $P_-$ and $P_+$ while
$$\frac{V(P_+)+V(P_-)}{2} - V(M)\ge 4\eta>0,$$
then 
$$\frac{u(P_+)+u(P_-)}{2} - u(M)\ge 2\eta>0.$$
The conclusion follows.
\end{proof}

Therefore, following this idea, one usually needs to answer the following questions:
\begin{enumerate}
\item What is the strict  defect? Typically, this arises from a one-dimensional solution to the same equation adapted to the (half-)width of the domain, whereas different equations involve distinct mechanisms that generate the defect.

\item What equation does the error satisfy, or whether there are ordered barriers  available? This usually comes from the  estimate on the residual \eqref{eq:abstract-residual}, or a suitable barrier argument. 

\item Which exact solution or branch is selected? 

\item In which topology is the defect transferred, and what final scaling or domain cleanup is needed? 
\end{enumerate}

Wang \cite{W2014} answered these questions by a direct nonlinear sandwich.  The author in \cite{Z2026} applied a slowly modulated one-dimensional family together with a spectral barrier which is suitable for semilinear equations, while in \cite{LLMQ2026} it is necessary to firstly insert an additional corrector to the one-dimensional profiles before comparison. Moreover, we remark that, even though the construction in Hamel--Nadirashvili--Sire \cite{HNS2016} relies on selecting an appropriate branch and does not involve an approximate PDE residual, some underlying ideas are similar.

Let us further explain these four counterexamples in detail below.

\subsection{\cite{W2014}: Comparison around an almost-solution}

The most illustrative example of this concept was implicitly provided by Wang \cite{W2014}\footnote{In a personal communication, Wang told the author that what we present here was not his original idea. However, his counterexample does fit our philosophy. }, which we present first. For the positive superlevel orientation, consider
$$
 \begin{cases}
 -\diver\left(\dfrac{Du}{\sqrt{1+|Du|^2}}\right)= 1&\text{in }\Om,\\
 u>0&\text{in }\Om,\\
 u=0&\text{on }\partial\Om.
 \end{cases}
$$
On a regular superlevel, recall \eqref{eq:mc-balance} gives
$$
 \frac{qH}{\sqrt{1+q^2}}
 -\frac{u_{nn}}{(1+q^2)^{3/2}}=1. 
$$
Then the uniform convexity of the boundary still gives a strictly convex collar neighborhood in a uniformly convex domain.   

Following the original negative variable in \cite{W2014}, we write
$$
 \widetilde u=-u<0,
 \qquad
 \diver\left(\frac{D\widetilde u}{\sqrt{1+|D\widetilde u|^2}}\right)=1.
$$
Then $\{u>t\}=\{\widetilde u<-t\}$, and a nonconvex sublevel of $\widetilde u$ is exactly the desired nonconvex superlevel of $u$.

Let us see how the defect enters locally. 
Up to smoothing only away from the selected bad level set, an almost-solution  can be written as
\begin{equation}\label{eq:Wang-profile}
 w(t,r)=\sqrt{1-\gamma(t)^2}-\sqrt{1-r^2}
\end{equation}
inside a convex domain bounded by $r=\pm\gamma(t)$; for each fixed $t$, this is indeed a $1$-dimensional solution to the equation.  Moreover, the level set
$$
 \mathcal L=\{w=-1/4\}
$$
has a fixed convex-hull gap.  Indeed, if $\mathcal C$ is the boundary of the convex hull of the relevant sublevel, set the deficit
\begin{equation}\label{eq:Wang-deficit}
     d(\mathcal L,\mathcal C)=\sup\{r_2-r_1:(t,r_1)\in \mathcal  L,(t,r_2)\in \mathcal C,r_2\geq r_1>0\}>0.
\end{equation}
Note that, by stretching only the long variable
$$
 \Om_\epsilon=\{(t,r):(\epsilon t,r)\in\Om\},
 \qquad w_\epsilon(t,r)=w(\epsilon t,r),
$$
the geometric defect is unchanged, while a direct calculation gives
\begin{equation}\label{eq:Wang-curvature-asymptotic}
 \diver\left(\frac{Dw_\epsilon}{\sqrt{1+|Dw_\epsilon|^2}}\right)
 =1+O(\epsilon^2).
\end{equation}

Conversely, we consider a  modification of \eqref{eq:Wang-profile} with larger  radius $\sqrt{1+\delta}$, where $\delta=\epsilon^{1/2}$. This yields a supersolution $v_\epsilon$.  Moreover, the Perron method, taking the envelope of subsolutions below an ordered supersolution, produces an exact zero-Dirichlet solution $\widetilde u_\epsilon$ with constant right-hand side $1-\epsilon$.  Now comparison principle tells
$$
 v_\epsilon\geq\widetilde u_\epsilon\geq w_\epsilon.
$$
The barriers converge, whereas the defect has a fixed positive value.  Consequently the exact level $L'_\epsilon$ satisfies
$$
 d(L'_\epsilon,C'_\epsilon)
 \geq\frac12d(L,C)>0
$$
for small $\epsilon$.  Via a suitable dilation, one restores the constant $1$, and smooth uniformly convex approximation of the domain appropriately  preserves the strict finite defect as well.

The important scale distinction is that, the PDE residual of the almost-solution is $O(\epsilon^2)$ by \eqref{eq:Wang-curvature-asymptotic}, whereas the explicit supersolution is separated from the almost by the larger quantity $O(\epsilon^{1/2})$.  Nevertheless, no sharp inverse estimate is needed here, since the geometric defect $d(L,C)$ is fixed independently of $\epsilon$; indeed, any comparison error tending to zero is eventually smaller than that deficit.   

\subsection{\cite{Z2026}: Thin domain and spectral barriers}
Roughly speaking, the author used the same  almost-solution philosophy as Wang in \cite{W2014}, while the mechanism is more structured toward stable solutions with superlinear nonlinearities. Specifically, in thin domains and for sufficiently low dimension,  stable solutions remain uniformly close to the corresponding lower-dimensional profile; see e.g.\ \cite{D1993}.

Recall that, for the semilinear Dirichlet problem
$$
 \begin{cases}
 -\Delta u=f(u)&\text{in }\Om,\\
 u>0&\text{in }\Om,\\
 u=0&\text{on }\partial\Om,
 \end{cases}
$$
a solution is \emph{minimal} if it lies below every other positive solution of the same problem.  It is \emph{strictly stable} when the first Dirichlet eigenvalue of $-\Delta-f'(u)$ is positive, equivalently when there is $\nu>0$ such that
$$
 \int_{\Om}|D\phi|^2-\int_{\Om}f'(u)\phi^2
 \geq \nu\int_{\Om}\phi^2
 \qquad \text{ for any} \ \phi\in H_0^1(\Om).
$$
When $f$ is positivity, increasing and strictly convex, one can show that stable solutions are minimal. 
These are branch and spectral properties, which  do not themselves impose a sign on the curvature of every level.

The mechanism to create a defect for an almost-solution is based on a one-dimensional fold.
Let $U_A$ solve
$$
 -U_A''=f(U_A),
 \qquad U_A(0)=A,
 \qquad U_A'(0)=0,
$$
with its first zero at $h(A)$, i.e. $U_A(h(A))=0$ and $U_A(t)>0$ for $0\le t<h(A)$, which is the \emph{half-width} where the solution is positive.  By setting $F(s)=\int_0^sf(\tau)\,d\tau$, one gets from the fundamental theorem of calculus together with integration by parts that
$$
 h(A)=\int_0^A\frac{ds}{\sqrt{2(F(A)-F(s))}}.
$$
A nondegenerate \emph{fold} on the lower stable branch satisfies  $h'(A_c)=0,$ $ h''(A_c)<0$, 
\begin{equation}\label{eq:fold}
\lambda_1\left(-\frac{d^2}{d\eta^2}-f'(U_{A_c})\right)=0
 \ \text{ and } \ 
 \lambda_1\left(-\frac{d^2}{d\eta^2}-f'(U_A)\right)>0
 \quad \text{ when } \ A<A_c\text{ near } A_c.
\end{equation}
For a fixed level $t>0$, let $Y_t(A)$ satisfy $U_A(Y_t(A))=t$.  Then one can show that, along the inverse lower branch $A_-(h)$,
\begin{equation}\label{eq:fold-level-expansion}
 Y_t(A_-(h))
 =Y_c-c_t\sqrt{h_c-h}+o(\sqrt{h_c-h}),
 \qquad c_t>0.
\end{equation}

Now we work on the main body of convex set, which is called slow channel in \cite{Z2026} and generates a formal finite midpoint gap.
Choose a concave function representing the half-width  
$$
 H(X)=h_c-D(X),
$$
where $D>0$ is convex while $\sqrt D$ is strictly concave at three points $X_-<X_0<X_+$ with $X_0=(X_-+X_+)/2$.  Then our formal one-dimensional profile is
$$
 V(X,\eta)=U_{A_-(H(X))}(\eta),
$$
and its upper level height is $Y(X)=Y_t(A_-(H(X)))$.  Now formula \eqref{eq:fold-level-expansion} yields a strict gap
\begin{equation}\label{eq:V}
     G:=\frac{Y(X_-)+Y(X_+)}2-Y(X_0)>0.
\end{equation}
This is the localized defect; compare with \eqref{eq:Wang-deficit} in the case of Wang. 

However, the profile $V$ is not yet a solution of the two-dimensional Dirichlet problem. 
Now we transfer from the profile $V$ to a real solution. 
We first attach smooth uniformly convex caps to the long channel and dilate the slow coordinate to the physical coordinate via $x_1=X/\epsilon$; recall a similar step is done also in Wang \cite{W2014}. Let $\Om_\epsilon$ be the resulting bounded uniformly convex domain.  Choose a stable one-dimensional profile $U=U_{h_*}$ on a slightly wider strip.  Then $0$ and $U(x_2)$ are respectively a subsolution and a supersolution.  Standard monotone iteration  starting from the subsolution $0$ gives the minimal exact solution $u_\epsilon$, with
$$
 0<u_\epsilon(x_1,x_2)\leq U(x_2).
$$
If
$$
 \nu_*=\lambda_1\left(-\frac{d^2}{d\eta^2}-f'(U),(-h_*,h_*)\right)>0,
$$
vertical slicing gives the quantitative stability estimate
\begin{equation}\label{eq:nustar}
\int_{\Om_\epsilon}
 \left(|D\phi|^2-f'(u_\epsilon)\phi^2\right)
 \geq \nu_*\int_{\Om_\epsilon}\phi^2.
\end{equation}
Thus minimality and strict stability pass to the exact realized branch, not merely to the almost-solution.

It remains to estimate the residual by a spectral barrier argument, which corresponds to the sandwich argument of Wang \cite{W2014}.
Set
$$
 v_\epsilon(X,\eta)=u_\epsilon(X/\epsilon,\eta),
 \qquad W_\epsilon=v_\epsilon-V.
$$
Since $V$ solves the one-dimensional equation only in $\eta$, the error satisfies the exact linearized equation
$$
 -\epsilon^2(W_\epsilon)_{XX}-(W_\epsilon)_{\eta\eta}
 -c_\epsilon(X,\eta)W_\epsilon
 =\epsilon^2V_{XX},
$$
where the mean-value coefficient obeys
$$
 0\leq c_\epsilon(X,\eta)\leq f'(U(\eta)).
$$

Let $\phi_*>0$ be the first eigenfunction of $-d^2/d\eta^2-f'(U)$ with eigenvalue $\nu_*>0$.  On a fixed central subchannel, a barrier of the form
$$
 \Phi(X,\eta)=C_0\epsilon^2\phi_*(\eta)
 +C_1\bigl(e^{-\kappa(X-X_L)/\epsilon}
          +e^{-\kappa(X_R-X)/\epsilon}\bigr)\phi_*(\eta),
 \qquad 0<\kappa^2<\nu_*,
$$
dominates both  $\epsilon^2V_{XX}$ and the artificial side-boundary data.  Then the maximum principle applied to $\Phi\pm W_\epsilon$ gives 
\begin{equation}\label{eq:Zhang-C0}
 \|v_\epsilon-V\|_{L^\infty(Q_{I'})}
 \leq C\epsilon^2,
\end{equation}
where $C$ is independent of $\epsilon$.
The spectral gap \eqref{eq:nustar} also provides, in the current argument, the coercivity required to invert the residual operator. This is the main methodological difference from Wang's direct nonlinear sandwich estimate.

Eventually we can transfer the finite defect to the  level sets. 
Since  the selected level $t>0$ is uniformly separated from zero and  $\partial_\eta V$ is also uniformly away from $0$, one can convert \eqref{eq:Zhang-C0} into the upper-height estimate, where $Y_\epsilon$ locally represents the level sets $\{v_\epsilon=t\}$ of the real solution $v_\epsilon$. That is, there exists another constant $C'>0$  independent of $\epsilon$ so that
$$
 |Y_\epsilon(X_j)-Y(X_j)|\leq C'\epsilon^2,
 \qquad j\in\{-,0,+\}.
$$
Therefore, one gets from \eqref{eq:V} that
$$
 \frac{Y_\epsilon(X_-)+Y_\epsilon(X_+)}2-Y_\epsilon(X_0)
 \geq G-2C'\epsilon^2>0
$$
for small $\epsilon$.  Thus if the superlevel sets of $v_\epsilon$ were convex, its upper height would be concave, which leads to a contradiction.  The final domain is already smooth and uniformly convex, and the exact solution is already minimal and strictly stable.

\subsection{\cite{LLMQ2026}: Corrector-enhanced barriers}

The translator construction belongs to the same comparison family, but it begins one step earlier: The naive one-dimensional profile has the wrong geometry and its leading residual must first be canceled.  

To be more precise, the unique negative solution of the following equation is considered in  \cite{LLMQ2026}: 
\begin{equation}\label{eq:translator-negative}
 \begin{cases}
 \displaystyle
 \diver\left(\frac{Du}{\sqrt{1+|Du|^2}}\right)
 =\frac1{\sqrt{1+|Du|^2}}&\text{in }\Om,\\
 u=0&\text{on }\partial\Om.
 \end{cases}
\end{equation}
The natural objects are sublevels $\{u<c\}$ since $u<0$.  Setting $v=-u>0$ converts them to superlevels and puts the equation in the form \eqref{eq:prototype} with $a(q)=(1+q^2)^{-1/2}$ and $\mathcal F(u,q)=(1+q^2)^{-1/2}$.  

Let us start with the one-dimensional profile to the corresponding solution \eqref{eq:translator-negative}:
$$
 U_0(x,y)=-\log\cos(y-c(x))+\log\cos A(x).
$$
This function is bad in the sense that, if the channel $c(x)-A(x)<y<c(x)+A(x)$ is also convex, every sublevel of $U_0$ is convex, which cannot generate the desired defect.  This is why a corrector is mathematically necessary. 

The corrector is constructed as follows. One would like to construct an almost-solution of the form 
$$U_0(x,y)+\tau w(\gamma_q(x),y).$$
Let $q>0$ be small and set
$$
 \gamma_0=\arccos q,
 \qquad \lambda=q^{9/4},
 \qquad \tau=\lambda^2=q^{9/2},
 \qquad \gamma_q(x)=\gamma_0+\lambda x-q^7x^2.
$$
With $g(s)=-\log\cos s$ and $A(s)=s\tan s+\log\cos s$, a direct calculation with the  boundary condition gives an explicit corrector $w(\gamma,y)$, which solves
$$
 w_{yy}-2\tan y\,w_y=\sec^2y+\tan^2\gamma,
 \qquad w(\gamma,\pm\gamma)=0,
$$
and eventually yields the corrected almost-solution
$$
 V_q(x,y)=g(y)-g(\gamma_q(x))+\tau w(\gamma_q(x),y).
$$
The corrector cancels the completing term linear in $\tau$, and then one can find points $P_+,P_-$ and their midpoint $M_q$ such that
$$
 V_q(P_\pm)\leq c_*-c_2q^{3/2},
 \qquad  V_q(M_q)\geq c_*+c_2q^{3/2}.
$$
On the other hand, via a delicate barrier argument, one can get
$$
 |u_q-V_q|
 \leq Cq^2+C|\log q|e^{-cq^{-1/4}}
 =o(q^{3/2}).
$$
Thus one can apply (a sublevel version of) Lemma~\ref{lem:finite-defect-transfer} to conclude the theorem. 

In the end, we remark that,  the corrector is essential to create an almost-solution with the desired finite defect and simultaneously lowers the weighted residual to a scale that the differential barriers can dominate.  Only after this cancellation does the comparison estimate become stronger than the defect scale $q^{3/2}$. 

\subsection{\cite{HNS2016}: A branch selection}

This last counterexample is deliberately placed after the perturbative constructions above, since it does not begin with an almost-solution.  Hamel--Nadirashvili--Sire \cite{HNS2016} indeed constructed one convex-domain problem with two positive solutions of opposite level geometry. Nevertheless, there are some common underlying ideas. 

To be more specific, let $\Om_a\subset\R^2$ be a smooth convex stadium whose central part becomes longer as $a\to\infty$.  The long direction is again $x$ and the transverse direction is $y\in(-1,1)$.  Symmetry and the moving-plane method make the selected solutions even in $x,y$ and decreasing in $|x|,|y|$.

Now one selects the exact branch under a constraint.
On $H_0^1(\Om_a)$ set
$$
 I_a(w)=\frac12\int_{\Om_a}|Dw|^2\,dx-\int_{\Om_a}w\,dx
$$
and impose
\begin{equation}\label{eq:HNS-constraint}
 \mathcal U_a
 =\left\{w\in H_0^1(\Om_a):\int_{\Om_a}g(w)\,dx=1\right\},
\end{equation}
where $g=0$ on $(-\infty,1]$ and $g=1$ on $[2,\infty)$.  A minimizer $u_a$ exists.  The Lagrange multiplier is nonzero and, after the sign is fixed by comparison with torsion, is positive; thus $u_a$ is a classical solution
$$
 -\Delta u_a=f_a(u_a),
 \qquad f_a(s)=1+\mu_a g'(s),
 \qquad \mu_a>0.
$$
 
They next show that, the same equation has a good branch. Indeed, let $v_a$ be the torsion function
$$
 -\Delta v_a=1,
 \qquad v_a=0\text{ on }\partial\Om_a.
$$
The range estimate keeps $v_a\leq1/2$, so $g'(v_a)=0$ and $f_a(v_a)=1$.  Hence $u_a$ and $v_a$ solve the same equation with the same domain and boundary data, and the known square-root concavity makes $v_a$ quasiconcave.

Subsequently,  they also localize the solution via a one-dimensional limit.
The constraint in \eqref{eq:HNS-constraint} implies that the active set
$$
 E_a:=\{u_a>1\}
$$
has horizontal half-length bounded independently of $a$, while there is a lower bound $C_y>0$ on the transverse half-width.  Away from this central  region and away from the caps, one has $g'(u_a)=0$, and  $u_a$ solves the torsion equation there.  Translation to a far point in the long strip, uniform interior Schauder estimates together with elliptic compactness  give a convergence
$$
 u_a(x,y)\to  \phi_0(y):=\frac{1-y^2}{2}
$$
which is uniform  on fixed compact subsets of the infinite strip; in this step, a Liouville-type result is established. We recall again \cite{D1993} for a version on stable solutions.

Eventually, one choose the aligned points
$$
 P=(0,C_y),
 \qquad Q_a=\left(\frac a4,\frac{C_y}{2}\right),
 \qquad R_a=\left(\frac a2,0\right),
 \qquad Q_a=\frac{P+R_a}{2}.
$$
By the definition of $C_y$ and symmetry, one has $u_a(P)>1$, and the strip limit gives
$$
 u_a(Q_a)\to  \frac12-\frac{C_y^2}{8},
 \qquad
 u_a(R_a)\to \frac12.
$$
By setting
$$
 \frac12-\frac{C_y^2}{8}<\lambda<\frac12.
$$
for large $a$, one gets
$$
 u_a(Q_a)<\lambda<u_a(R_a)<1<u_a(P).
$$
Thus $P,R_a\in\{u_a>\lambda\}$ while their midpoint $Q_a$ is not; this idea  was also used later in \cite{Z2026}.  Again the nonconvexity is transferred via a finite value defect with a strict positive value, even though this method has no residual estimate.

\end{document}